\documentclass{article}

\usepackage{enumitem}
\usepackage[T1]{fontenc}
\usepackage[utf8]{inputenc}
\usepackage[english]{babel}
\usepackage{amsfonts}
\usepackage{amsthm}
\usepackage{amsmath}
\usepackage{amssymb}
\usepackage{amstext}
\usepackage{cite}
\usepackage{url}
\usepackage{mathtools}
\usepackage{geometry}
\usepackage[hidelinks]{hyperref}

\newtheorem{theorem}{Theorem}[section]
\newtheorem{corollary}[theorem]{Corollary}
\newtheorem{lemma}[theorem]{Lemma}
\newtheorem{proposition}[theorem]{Proposition}
\newtheorem{q}[theorem]{Question}
\theoremstyle{definition}
\newtheorem{definition}[theorem]{Definition}
\newtheorem{remark}[theorem]{Remark}
\newtheorem{example}[theorem]{Example}

\newtheorem{deff}{Definition}
 
\newcommand{\Q}{\mathbb{Q}}

\newcommand{\OO}{\mathcal{O}}
\newcommand{\Z}{\mathbb{Z}}
\newcommand{\F}{\mathbb{F}}

\newcommand{\p}{\mathcal{P}}
\newcommand{\N}{\mathbb{N}}

\newcommand{\abs}[1]{\left|#1\right|}
\newcommand{\sign}{\text{Sign}}

\def\ord{\operatorname{ord}}
\def\min{\operatorname{min}}
\def\lcm{\operatorname{lcm}}

\def\defi{\coloneqq}

\title{A recurrence relation for elliptic divisibility sequences}
\author{Matteo Verzobio}
\date{}
\begin{document}
	\maketitle
	\renewcommand{\thefootnote}{}
	
	
	
	\renewcommand{\thefootnote}{\arabic{footnote}}
	\setcounter{footnote}{0}
	\begin{abstract}
		In literature, there are two different definitions of elliptic divisibility sequences. The first one says that a sequence of integers $\{h_n\}_{n\in \N}$ is an elliptic divisibility sequence if it satisfies the recurrence relation $h_{m+n}h_{n-m}h_{r}^2=h_{n+r}h_{n-r}h_{m}^2-h_{m+r}h_{m-r}h_{n}^2$ for all natural numbers $n\geq m\geq r$. The second definition says that a sequence of integers $\{\beta_n\}_{n\in \N}$ is an elliptic divisibility sequence if it is the sequence of the square roots (chosen with an appropriate sign) of the denominators of the abscissas of the iterates of a point on a rational elliptic curve. It is well-known that the two sequences are not equivalent. Hence, given a sequence of the denominators $\{\beta_n\}_{n\in \N}$, in general the relation $\beta_{m+n}\beta_{n-m}\beta_{r}^2=\beta_{n+r}\beta_{n-r}\beta_{m}^2-\beta_{m+r}\beta_{m-r}\beta_{n}^2$ does not hold for all $n\geq m\geq r$. We will prove that the recurrence relation above holds for $\{\beta_n\}_{n\in \N}$ under some conditions on the indexes $m$, $n$, and $r$.
	\end{abstract}
	\section{Introduction}
	The goal of this paper is to make a remark on the definition of elliptic divisibility sequences (also called EDS). In literature, there are two different definitions of EDS. We want to show the link between these two definitions. 
	
	The first definition is due to Ward, in \cite{ward}. It is completely arithmetical.
	\begin{deff}\label{defward}
		A sequence of integers $\{h_n\}_{n\in \N}$ is an \textbf{elliptic divisibility sequence} if it satisfies the following properties:
		\begin{itemize}
			\item $h_0=0$;
			\item $h_1=1$;
			\item $h_2$ divides $h_4$;
			\item for all $n\geq m\geq r$,
			\begin{equation}\label{eqhn}
				h_{m+n}h_{n-m}h_{r}^2=h_{n+r}h_{n-r}h_{m}^2-h_{m+r}h_{m-r}h_{n}^2.
			\end{equation}
		\end{itemize}
	\end{deff}
	
	In literature, there is an other definition of elliptic divisibility sequences. This definition is more geometrical.
	\begin{deff}\label{defeds}
		Let $E$ be a rational elliptic curve defined by a Weierstrass equation with integer coefficients, and let $P\in E(\Q)$. For every $n\in \N$ write 
		\[
		x(nP)=\frac{A_n(E,P)}{B_n^2(E,P)}
		\]
		with $A_n(E,P)$ and $B_n(E,P)$ two coprime integers and $B_n(E,P)\geq0$. If $nP=O$, the identity of the curve, then we put $B_n(E,P)=0$. Let $\psi_n$ be the $n$-th division polynomial of $E$, as defined in Definition \ref{defdivpol}. Define
		\[
		\beta_n(E,P)=\sign(\psi_n(x(P),y(P)))\cdot \frac{B_n(E,P)}{B_1(E,P)}.
		\] 
		
		We say that the sequence $\{\beta_n(E,P)\}_{n\in \N}$ is an \textbf{elliptic divisibility sequence}.
	\end{deff}
	\begin{remark}
		One can easily show that $B_1(E,P)$ divides $B_n(E,P)$ for every $n$ and then the sequence of the $\beta_n(E,P)$ is a sequence of integers.
		
		The fact that the denominator of $x(nP)$ is a square follows from the fact that the coefficients of the Weierstrass equation are integers.
	\end{remark}
	\begin{remark}
		The sequences $\{\beta_n(E,P)\}_{n\in \N}$ and $\{B_n(E,P)\}_{n\in \N}$ are clearly strictly related. In some papers the sequence of the $B_n$ is studied instead of the sequence of the $\beta_n$. We will consider the sequences of the $\beta_n(E,P)$ since it is easier to relate them with the sequences of Definition \ref{defward}. In the definition of $\beta_n$, we divide by $B_1$ in order to have the additional property $\beta_1(E,P)=1$. Finally, the choice of the sign of $\beta_n$ is necessary in order to link the two definitions of EDS. The reason for this choice will become clear during the paper.
	\end{remark}
	
	The study of the elliptic divisibility sequences is very interesting and has applications in a lot of fields, as for example cryptography or logic. These sequences are divisibility sequences. Recall that a sequence of integers $\{a_n\}_{n\in \N}$ is a divisibility sequence if
	\[
	m\mid n\Longrightarrow a_m\mid a_n.
	\] 
	
	\begin{definition}
		Given a sequence of integers $\{a_n\}_{n\in \N}$, we say that the sequence is an \textbf{EDSA} if it is an elliptic divisibility sequence as in Definition \ref{defward}. We say that the sequence is an \textbf{EDSB} if it is an elliptic divisibility sequence as in Definition \ref{defeds}.
	\end{definition}
	There exist some sequences that are both EDSA and EDSB. Anyway, it is easy to show that the two definitions are not equivalent. The goal of this paper is to show the relation between the two definitions.
	
	First of all, we show an example of a sequence that is both an EDSA and an EDSB.
	\begin{example}
		Consider the EDSA $h_1=1$, $h_2=2$, $h_3=-1$, and $h_4=-36$. Observe that once one knows the value of $h_i$ for $i\leq 4$, then every term can be computed using the recurrence relation (\ref{eqhn}). Take $E$ the elliptic curve defined by the equation $y^2=x^3+x+1$ and $P=(0,1)$. Computing the first terms, we obtain $\beta_1(E,P)=1$, $\beta_2(E,P)=2$, $\beta_3(E,P)=-1$, and $\beta_4(E,P)=-36$. For example, since $2P=(1/4,-9/8)$ and $\psi_2(x,y)=2y$, we have
		\[
		\beta_2(E,P)=\sign(\psi_2(x(P),y(P)))\frac{B_2(E,P)}{B_1(E,P)}=\sign(2y(P))\frac{\sqrt{4}}{\sqrt{1}}=2.
		\] 
		Hence, for $i\leq 4$, we have $\beta_i(E,P)=h_i$. Using the work in the next pages and in particular Theorem \ref{main}, it is possible to show that in general $h_n=\beta_n(E,P)$.
	\end{example}
	Anyway, in general it is not true that every EDSB is an EDSA. In the same way, it is not true that every EDSA is an EDSB. We show two examples of this fact.
	\begin{example}
		Consider the sequence $a_n=n$. This is an EDSA, by direct computation. We can easily show that this sequence is not an EDSB. Suppose, by absurd, that there exists $E$ and $P$ such that $\beta_n(E,P)=n$ for every $n\in \N$. Since $a_n\neq 0$ for $n\geq 1$, we have that $P$ is a non-torsion point. Using \cite[Example IX.3.3]{arithmetic}, we have that
		\[
		\lim_{n\to \infty} \frac{\log \abs{\beta_n(E,P)}}{n^2}=c>0.
		\]
		The constant $c$ depends on $E$ and $P$. Anyway, it is always strictly positive, if $P$ is a non-torsion point.
		Observe that
		\[
		\lim_{n\to \infty} \frac{\log n}{n^2}=0
		\]
		and then $\{n\}_{n\in \N}$ cannot be an EDSB.
	\end{example}
	\begin{example}\label{ex2}
		Let $E$ be the elliptic curve defined by the equation $y^2=x^3+x+6$ and take the point $P=(-1,2)$ in $E(\Q)$. Consider the EDSB $\{\beta_n(E,P)\}_{n\in \N}$. One can compute that $\beta_1(E,P)=1$, $\beta_2(E,P)=1$, $\beta_3(E,P)=-1$, $\beta_4(E,P)=-3$, and $\beta_5(E,P)=1$. This is not an EDSA since it does not satisfy (\ref{eqhn}). Indeed, if we put $n=3$, $m=2$, and $r=1$ in (\ref{eqhn}) we should have
		\[
		\beta_5\beta_1^3=\beta_4\beta_2^3-\beta_1\beta_3^3.
		\]
		This is not true and then the sequence is not an EDSA.
	\end{example}
	The problem of understanding when an EDSA is an EDSB has been studied in \cite[Section IV]{ward}. We will give some details on this problem at the beginning of Section \ref{secmain}.
	
	Instead, we will focus on the problem of understanding when an EDSB is an EDSA. This problem has been studied in \cite{shipsey}.
	\begin{theorem}\cite[Theorem 5.1.1]{shipsey}\label{sh}
		Let $E$ be an elliptic curve defined by a Weierstrass equation with integer coefficients and let $P\in E(\Q)$ be a non-torsion point. There exists a multiple $Q$ of $P$ such that $\{\beta_n(E,Q)\}_{n\in \N}$ is an EDSA.
	\end{theorem}
	The goal of this paper is to answer the following question. 
	\begin{q}
		Given an EDSB that it is not an EDSA, how far is the sequence from being an EDSA?
	\end{q}
	We know that $\beta_0=0$, that $\beta_1=1$, and that $\beta_2$ divides $\beta_4$ since it is a divisibility sequence. So, if the sequence is not an EDSA, then Equation (\ref{eqhn}) does not hold.
	
	We want to show that every EDSB satisfies a subset of the equations in (\ref{eqhn}). Indeed, we will prove the following theorem.
	\begin{theorem}\label{main}
		Let $E$ be an elliptic curve defined by a Weierstrass equation with integer coefficients and let $P\in E(\Q)$ be a non-torsion point. Consider the EDSB $\{\beta_n\}_{n\in \N}=\{\beta_n(E,P)\}_{n\in \N}$. Let $n\geq m\geq  r$ be three positive integers such that two of them are multiples of $M(P)$, a constant that we will define in Definition \ref{MP}. Then,
		\[
		\beta_{n+m}\beta_{n-m}\beta_r^2=\beta_{n+r}\beta_{n-r}\beta_m^2-\beta_{m+r}\beta_{m-r}\beta_n^2.
		\]
	\end{theorem}
	As we will explain in Remark \ref{rem}, this theorem is a generalization of Theorem \ref{sh}.
	
	As we noticed before, if $\{h_n\}_{n\in \N}$ is an EDSA and we know $h_i$ for $i \leq 4$, then we can compute $h_k$ for every $k\in \N$, using the recurrence relation. Thanks to Theorem \ref{main} we know that, if $\{\beta_n(E,P)\}_{n\in \N}$ is an EDSB and we know $\beta_i(E,P)$ for $i \leq 4M(P)$, then we can compute $\beta_k(E,P)$ for every $k\in \N$. Indeed, if $k>4M(P)$, then we put $r=M(P)$, $m=2M(P)$, and $n=k-2M(P)$ and using the recurrence relation we can compute $\beta_k(E,P)$ using the induction.
	\begin{remark}
		We used the hypothesis that $E$ is defined by a Weierstrass equation with integer coefficients because otherwise the sequence of the $\beta_n$ is not in general a sequence of integers. We will deal with the case of an elliptic curve defined by a Weierstrass equation without integer coefficients in Proposition \ref{nonint}.
	\end{remark}
	\begin{remark}\label{notmin}
		The hypothesis that $P$ is a non-torsion point is necessary in order to prove the results of Section \ref{secgcd}. Indeed, we are not able to prove the results of that section without this hypothesis.
		In Corollary \ref{remtors}, we will prove that Theorem \ref{main} holds in the case when $P$ is torsion point, if $M(P)=1$. For some more considerations on the EDSB in the case when $P$ is a torsion point, see \cite[Section 5.5]{shipsey}. Observe that, in the case when $P$ is a torsion point, the sequence $\beta_n(E,P)$ is quite simple. Indeed, in this case, the sequence is periodic with order small. 
	\end{remark}
	\begin{remark}
		Sometimes, in Definition \ref{defward}, Equation (\ref{eqhn}) is given only for $r=1$. Anyway, as is proved in \cite[Lemma 30.1]{ward}, the two definitions are equivalent.
	\end{remark}
	\begin{remark}
	Let $R$ be a principal ideal domain with field of fractions $K$. One can prove an analogue of Theorem \ref{main} for elliptic divisibility sequences defined over $K$. For more details, see \cite[Corollary 4.1]{naskrkecki2022common}.
	\end{remark}
	\section{Division polynomials}\label{divpol}
	The aim of this section is to introduce the division polynomials. Even if our main theorem is for elliptic curves defined over $\Q$, in the next two sections we will work with elliptic curves defined over a number field $K$, in order to give the most general results. We denote with $\OO_K$ the ring of integers of $K$. Define $M_K^0$ as the set of all finite places of $K$. Given $\nu\in M_K^0$, we define $K_\nu$ as the completion of $K$ with respect to $\nu$.
	
	Let $E$ be an elliptic curve defined by the equation
	\begin{equation}\label{wei}
		y^2+a_1xy+a_3y=x^3+a_2x^2+a_4x+a_6
	\end{equation}
	with coefficients in a number field $K$. 
	Define the quantities
	\begin{align}\label{bi}
		&b_2=4a_2+a_1^2,\\&b_4=2a_4+a_1a_3,\nonumber\\&b_6=a_3^2+4a_6,\nonumber\\&b_8=a_1^2a_6+4a_2a_6-a_1a_3a_4+a_2a_3^2-a_4^2.\nonumber
	\end{align}
	Given a point $P\in E(K)$ and $n\in \N$, we want to show how to effectively compute the coordinates of the point $nP$. In order to do so, we need to define the so-called
	\textbf{division polynomials}.
	\begin{definition}\label{defdivpol}
		Let $\psi_n\in\Z[x,y,a_1,a_2,a_3,a_4,a_6]$ be the sequence of polynomials defined as follows:
		\begin{itemize}
			\item $\psi_0=0$;
			\item $\psi_1=1$;
			\item $\psi_2=2y+a_1x+a_3$;
			\item $\psi_3=3x^4+b_2x^3+3b_4x^2+3b_6x+b_8$;
			\item $\psi_4=\psi_2(2x^6+b_2x^5+5b_4x^4+10b_6x^3+10b_8x^2+(b_2b_8-b_4b_6)x+(b_4b_8-b_6^2))$;
			\item $\psi_{2n+1}=\psi_{n+2}\psi_n^3-\psi_{n-1}\psi_{n+1}^3 \text{  for  }n\geq 2$;
			\item $\psi_{2n}\psi_2=\psi_n\psi_{n+2}\psi_{n-1}^2-\psi_n\psi_{n-2}\psi_{n+1}^2 \text{  for  }n\geq 3$.
		\end{itemize}
		Recall that the coefficients $b_i$ are defined in (\ref{bi}) and depend only on the coefficients $a_i$. These polynomials are the so-called \textbf{division polynomials}. For $n\geq 1$, define also the polynomials
		\begin{equation}\label{phi}
			\phi_n=x\psi_n^2-\psi_{n+1}\psi_{n-1}.
		\end{equation}
	\end{definition}
	Observe that the points on the curve satisfy the equation
	\begin{equation}\label{eqbi}
		(2y+a_1x+a_3)^2=4x^3+b_2x^2+b_4x+b_6.
	\end{equation}
	This can be proved by substituting the coefficients $b_i$ with the coefficients $a_i$ using the definitions given in (\ref{bi}) and obtaining the Weierstrass equation (\ref{wei}) that defines the curve.
	
	We will evaluate the polynomials of Definition \ref{defdivpol} only on points of the curve. For such points, it holds (\ref{eqbi}) and so in the polynomials we can substitute $(2y+a_1x+a_3)^2$ with $4x^3+b_2x^2+b_4x+b_6$.
	
	For example,
	\[
	\psi_2^2=(2y+a_1x+a_3)^2=4x^3+b_2x^2+b_4x+b_6.
	\]
	\begin{lemma}\label{propdivpol}
		Fix $n\geq 1$.
		\begin{itemize}
			\item Using the substitution 
			\[
			(2y+a_1x+a_3)^2=4x^3+b_2x^2+b_4x+b_6,
			\]
			we can assume that the polynomial $\phi_n$ does not depend on $y$. Therefore, the polynomial $\phi_n$ is in $\Z[x,a_1,a_2,a_3,a_4,a_6]$.
			\item If $n$ is odd, then the polynomial $\psi_n$ is in $\Z[x,a_1,a_2,a_3,a_4,a_6]$. Instead, if $n$ is even, then $\psi_n$ is a polynomial in $\Z[x,a_1,a_2,a_3,a_4,a_6]$, multiplied by $(2y+a_1x+a_3)$. Therefore, using $(2y+a_1x+a_3)^2=4x^3+b_2x^2+b_4x+b_6$, we can assume that $\psi_n^2$ does not depend on $y$. So, 
			\[
			\psi_n^2\in \Z[x,a_1,a_2,a_3,a_4,a_6]
			\]
			for every $n$.
		\end{itemize}
	\end{lemma}
	\begin{proof}
		See \cite[Exercise 3.7]{arithmetic}.
	\end{proof}
	\begin{lemma}\label{propdivpol2}
		If the curve $E$ is fixed, then the coefficients $a_i$ are fixed. Therefore, we say that $\phi_n(x)$ and $\psi_n^2(x)$ depend only on $x$.
		\begin{itemize}
			\item The polynomial $\phi_n(x)$ is monic and has degree $n^2$.
			\item The polynomial $\psi_n^2(x)$ has degree $n^2-1$ and its leading coefficient is $n^2$. The zeros of this polynomial are the $x$-coordinates of the non-trivial $n$-torsion points of $E(\overline{\Q})$.
			\item For every $P\in E(K)$ that is not a $n$-torsion point, we have 
			\[
			x(nP)=\frac{\phi_n(x(P))}{\psi_n^2(x(P))}.
			\]  	
		\end{itemize}
	\end{lemma}
	\begin{proof}
		See \cite[Exercise 3.7]{arithmetic}.
	\end{proof}
	\begin{example}
		Let $E$ be the elliptic curve defined by the equation $y^2=x^3+x$. Then, by definition, $\psi_1=1$, $\psi_2=2y$, and $\psi_3=3x^4+6x^2-1$.
		Moreover,
		\[
		\phi_2=x\psi_2^2-\psi_3\psi_1=4xy^2-3x^4-6x^2+1.
		\]
		Using the substitution $y^2=x^3+x$, we obtain $\psi_2^2=4x^3+4x$ and $\phi_2=x^4-2x^2+1$. Therefore,
		\[
		x(2P)=\frac{\phi_2(x(P))}{\psi_2^2(x(P))}=\frac{x(P)^4-2x(P)^2+1}{4x(P)^3+4x(P)}.
		\]
	\end{example}
	\begin{remark}
		The sequence $\{\psi_n(x(P),y(P))\}_{n\in \N}$ is almost an EDSA. Indeed, it satisfies every condition of Definition \ref{defward}, except for the condition that the terms are integers. This follows from \cite[Exercise III.7.g]{arithmetic}.
	\end{remark}
	
	Suppose now that $E$ is defined over $\Q$.
	Consider $p$ a prime in $\Z$ and $\nu$ the place associated with $p$. Suppose that $\nu(x(P))\geq 0$. This happens if $P$ does not reduce to the identity modulo $p$. So, $\nu(\phi_n(x(P)))\geq 0$ and $\nu(\psi_n^2(x(P)))\geq 0$. Recall that, as we defined in the introduction, 
	\[
	x(nP)=\frac{A_n(E,P)}{B_n^2(E,P)}.
	\]
	Observe that
	\begin{align*}
		2\nu(B_n(E,P))&=\max\{0,-\nu(x(nP))\}\\&=\max\{0,\nu(\psi_n^2(x(P)))-\nu(\phi_n(x(P)))\}\\&=\nu(\psi_n^2(x(P)))-\min\{\nu(\psi_n^2(x(P))),\nu(\phi_n(x(P)))\}.
	\end{align*}
	From the previous equality, in order to study the sequence of the $\beta_n(E,P)=\pm B_n(E,P)/B_1(E,P)$, we need to study 
	\[
	\min\{\nu(\phi_n(x(P))),\nu(\psi_n^2(x(P)))\}.
	\]
	In the next section, we will study this quantity.
	\section{The sequence of the gcd}\label{secgcd}
	The goal of this section is to prove Proposition \ref{gmn}, that is necessary in order to prove Theorem \ref{main}. Again, we will work assuming that the curve $E$ is defined over a number field $K$, in order to give the most general results.
	
	Let $\nu\in M_K^0$, $K_\nu$ be the completion of $K$ with respect to $\nu$, and $\p$ be the prime associated with $\nu$. Let 
	\[
	E_0(K_\nu)=\{M\in E(K_\nu)\mid \overline{M} \text{  is not singular in  }E(\F_\p)\}
	\]
	where $\overline{M}$ is the reduction of the point $M$ in the reduced curve $E(\F_\p)$ modulo $\p$. With $\F_\p$ we denote the field $\OO_K/\p\OO_K$. This is a subgroup of $E(K_\nu)$ and $E(K_\nu)/E_0(K_\nu)$ is finite, thanks to \cite[Corollary C.15.2.1]{arithmetic}. 
	\begin{definition}
		Let $K$ be a number field, $\nu\in M_K^0$, and let $E$ be an elliptic curve defined by a Weierstrass equation with integer coefficients in $K_\nu$. Let $P\in E(K_\nu)$. Denote with $r(\p,P)$ the order of $P$ in $E(K_\nu)/E_0(K_\nu)$.
	\end{definition} 
	Recall that $\p$ is the prime associated with the finite valuation $\nu$.
	\begin{definition}\label{MP}
		Let $E$ be an elliptic curve defined by a Weierstrass equation with integer coefficients in $K$ and let $P\in E(K)$. Define
		\[
		M(P)\defi \lcm_{\p}\{r(\p,P)\}.
		\]
		If $\p$ does not divide $\Delta$, the discriminant of the curve, then $E$ has good reduction and therefore 
		\[
		E_0(K_\nu)=E(K_\nu).
		\]
		So, $r(\p,P)\neq 1$ only for finitely many $\p$ and then $M(P)$ is a well-defined positive integer.
	\end{definition}
	The value of $M(P)$ can be bounded using \cite[Corollary C.15.2.1]{arithmetic}. For example, if the $j$-invariant of the curve is integral, then $M(P)$ divides $12$. Observe that the point $Q=M(P)P$ is non-singular modulo every prime. Hence, $M(Q)=1$.
	\begin{definition}\label{defgnu}
		Let $E$ be an elliptic curve defined by a Weierstrass equation with integer coefficients in $K$ and let $P\in E(K)$. Let $\nu\in M_K^0$ and $n\in \N$. If $nP\neq O$, then define
		\[
		g_{n,\nu}(P)\coloneqq\min\{\nu(\phi_n(x(P))),\nu(\psi_n^2(x(P)))\}.
		\]
		Moreover, if $nP=O$, we put $g_{n,\nu}(P)=0$.
	\end{definition} 
	\begin{proposition}\label{ayad}
		Let $E$ be an elliptic curve defined by a Weierstrass equation with integer coefficients in $K$. Let $P\in E(K)$ and assume $\nu(x(P))\geq0$. If $r(\p,P)=1$, then 
		\[
		g_{n,\nu}(P)=0
		\]
		for every $n\in \N$.
	\end{proposition}
	\begin{proof}
		See \cite[Theorem A]{Ayad}.
	\end{proof}
	Thanks to the previous proposition, we know that $g_{n,\nu}(P)\neq0$ only if $E$ is singular modulo $\p$, assuming $\nu(x(P))\geq 0$. Hence, we need to compute $g_{n,\nu}(P)$ only in the case when $E$ is singular modulo $\p$, where $\p$ is the prime associated with $\nu$.
	We will show that the terms of the sequence $g_{n,\nu}(P)$ satisfies a recurrence relation.
	
	Now, we study $g_{n,\nu}(P)$ in the case $r(\p,P)>1$.
	\begin{theorem}\cite[Theorem 4]{chhadiv}\label{ch}
		Let $E$ be an elliptic curve defined by a Weierstrass equation with coefficients in $\OO_K$. Let $P$ be a non-torsion point of $E(K)$ and assume $\nu(x(P))\geq0$. Let $r=r(\p,P)>1$ and $n>0$. Then,
		\begin{equation}\label{eqch}
			g_{n,\nu}(P)=\begin{cases}
				\mu m^2 \text {  if  } n=mr,
				\\
				4\mu m^2 \pm 2(2\nu(\frac{\psi_k(x(P),y(P))}{\psi_{r-k}(x(P),y(P))})+\mu) m\\+2\nu(\psi_k(x(P),y(P))) \text{  if  } n=2mr\pm k \text{  with  } 1\leq k<r,
			\end{cases}
		\end{equation}
		where $\mu=g_{r,\nu}(P)$. 
	\end{theorem}
	\begin{remark}
		The previous theorem, as far as we know, is correct, even if the proof of this fact in \cite{chhadiv} has a gap. We briefly show why the proof of Theorem 4 in \cite{chhadiv} is wrong and how to fix the problem. At the beginning of the paper, it is claimed in Lemma 1 that if $P\in E_0(K_\nu)$ and $\nu(x(P))<0$, then $\nu(x(nP))=\nu(x(P))-2\nu(n)$ for every $n\geq 1$. This is not true, in general. For example, if we take $K_\nu=\Q_2$, the curve $E$ defined by the equation $y^2+xy=x^3+x^2-2x$, the point $P=(-1/4,7/8)$, and $n=2$, then the equation does not hold. This mistake affects Theorem 1 and Theorem 3 of the paper, that are false. Anyway, even if the proof of Theorem 4 uses Theorem 3, we can easily prove the theorem replicating the work of \cite{chhadiv}, with a little adjustment. 
		
		Observe that, if $P\in E_0(K_\nu)$ and $\nu(x(P))<0$, then $\nu(x(nP))\leq \nu(x(P))$ for every $n\geq 1$. This follows easily from Lemma \ref{propdivpol2}. In order to prove Theorem 4, one just needs to replicate the work in \cite{chhadiv} substituting Lemma 1 with the previous observation. Using this substitution,  one can prove an analogue of Theorem 1 and 3. With the new versions of these theorems, the proof of Theorem 4 still works. Therefore, \cite[Theorem 4]{chhadiv} is true.
		
		In the case when $E$ is in minimal form, there is a more explicit version of the Theorem proved in \cite[Theorem 1.1]{yv21}.
	\end{remark}
	\begin{proposition}\label{gcd}
		Let $E$ be an elliptic curve defined by a Weierstrass equation with integer coefficients in $K$ and let $P\in E(K)$ be a non-torsion point. Let $\nu$ be a finite place and $\p$ be the prime associated with $\nu$. Assume $\nu(x(P))\geq 0$. If $m$ is a multiple of $r(\p,P)$, then 
		\[
		g_{n+m,\nu}(P)+g_{\abs{n-m},\nu}(P)=2(g_{n,\nu}(P)+g_{m,\nu}(P)).
		\]
		If $m$ is a multiple of $M(P)$, then the equation holds for every $\nu$. 
	\end{proposition}
	\begin{proof}
		If $r(\p,P)=1$, then we conclude using Proposition \ref{ayad}. So, we assume $r(\p,P)>1$.
		
		If $n=m$, then $m=n=kr$, where $r=r(\p,P)>1$. So, using (\ref{eqch}),
		\[
		g_{n+m,\nu}(P)+g_{\abs{n-m},\nu}(P)=g_{2kr,\nu}(P)=4\mu k^2=2(g_{kr,\nu}(P)+g_{kr,\nu}(P)).
		\]
		In this case, the proposition holds. 
		
		Now, we assume that $n\neq m$. Define $n_1\defi\max\{n,m\}$ and $n_2\defi\min\{n,m\}$. So, $n_1>n_2$ and suppose that $n_2$ is a multiple of $r=r(\p,P)$. 
		If $n_1\equiv 0 \mod r$, then $n_1=m_1r$ and $g_{n_1,\nu}(P)=\mu m_1^2$. In the same way, $g_{n_2,\nu}(P)=\mu m_2^2$. Therefore, using Theorem \ref{ch}, \[g_{n_1+n_2,\nu}(P)=\mu (m_1+m_2)^2,\] and \[g_{n_1-n_2,\nu}(P)=\mu (m_1-m_2)^2.\] Now, it is easy to conclude observing that 
		\[
		(m_1+m_2)^2+(m_1-m_2)^2=2m_1^2+2m_2^2.
		\]
		
		So, we assume that $n_1\not\equiv 0 \mod r$.
		We put $n_1=2rm_1\pm k$ for $0<k<r$ and $n_2=2rm_2$ or $n_2=r(2m_2+1)$ since $n_2\equiv 0 \mod r$. Put $n_3=n_1+n_2$ and $n_4=n_1-n_2$. We want to study 
		\[
		L_{n_1,n_2}=g_{n_1+n_2,\nu}(P)+g_{n_1-n_2,\nu}(P)-2(g_{n_1,\nu}(P)+g_{n_2,\nu}(P)).
		\] 
		We want to prove that $L_{n_1,n_2}=0$.
		We are going to use Theorem \ref{ch}. For notational convenience, we write $\psi_k$ instead of $\psi_k(x(P),y(P))$. 
		\begin{enumerate}
			\item If $n_1=2m_1r+k$ and $n_2\equiv 0\mod(2r)$, then $n_3=2(m_1+m_2)r+k$ and $n_4=2(m_1-m_2)r+k$. So,
			\begin{align*}
				L_{n_1,n_2}=&4\mu(m_1+m_2)^2+2(2\nu(\frac{\psi_k}{\psi_{r-k}})+\mu)(m_1+m_2)+2\nu(\psi_k) \\&+ 4\mu(m_1-m_2)^2+2(2\nu(\frac{\psi_k}{\psi_{r-k}})+\mu) (m_1-m_2)+2\nu(\psi_k)\\&-2(4\mu m_1^2+2(2\nu(\frac{\psi_k}{\psi_{r-k}})+\mu) m_1+2\nu(\psi_k) )-2 (4\mu m_2^2)
				\\=&0.
			\end{align*}
			\item If $n_1=2m_1r-k$ and $n_2\equiv 0\mod(2r)$, then $n_3=2(m_1+m_2)r-k$ and $n_4=2(m_1-m_2)r-k$. Repeating the proof as in the case 1, we can show $L_{n_1,n_2}=0$.
			\item If $n_1=2m_1r+k$ and $n_2=(2m_2+1)r$, then
			$n_3=2(m_2+m_1+1)r-(r-k)$ and $n_4=2(-m_2+m_1)r-(r-k)$.  Repeating the proof as in the case 1, we can show $L_{n_1,n_2}=0$.
			\item If $n_1=2m_1r-k$ and $n_2=(2m_2+1)r$, then $n_3=2(m_1+m_2)r+(r-k)$ and $n_4=2(m_1-m_2-1)r+(r-k)$.  Repeating the proof as in the case 1, we can show $L_{n_1,n_2}=0$.
		\end{enumerate}
		In the case when $n_1$ is the multiple of $r$, the proof is identical. This concludes the first part of the proof.
		
		Assume now that $m$ is a multiple of $M(P)$. Then, it is a multiple of $r(\p,P)$ for every $\p$ and we conclude with the first part of the proposition.
	\end{proof}
	\begin{remark}\label{16}
		In the previous proposition, the hypothesis that $m$ is a multiple of $r(\p,P)$ is necessary. For example, take $E$ the elliptic curve defined by the equation $y^2=x^3+x+6$ and $P=(-1,2)\in E(\Q)$. Let $\nu$ be the place associated with $2$ and then $\nu(x)=\ord_2(x)$. By definition, $\phi_2(x)=x^4-2x^2-48x+1$ and $\psi_2^2(x)=4(x^3+x+6)$. So, by direct computation, 
		\begin{align*}
			g_{2,\nu}(P)&=\min\{\ord_2(\phi_2(x(P))),\ord_2(\psi_2^2(x(P)))\}\\&=\min\{\ord_2(48),\ord_2(16)\}\\&=4
		\end{align*}
		and, in the same way, $g_{3,\nu}(P)=12$. Putting $m=1$ and $n=2$ we have
		\begin{align*}
			g_{n+m,\nu}(P)+g_{n-m,\nu}(P)-2g_{m,\nu}(P)-2g_{n,\nu}(P)=&g_{3,\nu}(P)+g_{1,\nu}(P)-2g_{2,\nu}(P)-2g_{1,\nu}(P)\\=&g_{3,\nu}(P)-2g_{2,\nu}(P)\\=&4
		\end{align*}
		and then the equation of the previous proposition does not hold. Observe that 
		\[
		\frac{d(x^3+x+6)}{dx}\Bigr|_{x=-1}=3x^2+1\Bigr|_{x=-1}=4
		\]
		and so $P$ is singular modulo $2$ since
		\[
		\frac{d(x^3+x+6)}{dx}\Bigr|_{x=x(P)}\equiv\frac{d(y^2)}{dy}\Bigr|_{y=y(P)}\equiv 0\mod 2.
		\] 
		In the same way $2P=(3,-6)$ is singular. Instead, $3P=(2,4)$ is not singular since 
		\[
		\frac{d(x^3+x+6)}{d x}\Bigr|_{x=2}=3x^2+1\Bigr|_{x=2}\equiv 1\not\equiv 0 \mod{2}.
		\]
		So, in order to apply the previous proposition in this case we need to take $m$ multiple of $3$.
	\end{remark}
	From now on and until the end of the section, we will assume that $K=\Q$. We show why the study of the $g_{n,\nu}(P)$ is important for the study of the sequence $\beta_n(E,P)$.
	\begin{definition}\label{homo}
		Let $u$ and $v$ be two integers. Define $\phi_n(u,v)$ as the homogenization of $\phi_n(x)$ evaluated in $u$ and $v$, which is $v^{n^2}\phi_n(u/v)$. In the same way, define $\psi_n^2(u,v)$ as the homogenization of $\psi_n^2(x)$.
	\end{definition}
	\begin{remark}
		The sequence $\psi_n^2(u,v)$ is the square of an EDSA. This follows from \cite[Exercise III.7.g]{arithmetic}.
	\end{remark}
	\begin{example}
		Let $E$ be the elliptic curve defined by the equation $y^2=x^3+x$. Then, $\psi_2^2=4x^3+4x$ and $\phi_2=x^4-2x^2+1$. So,
		\[
		\phi_2(u,v)=v^4\phi_2\Big(\frac uv\Big)=v^4\Big[\Big(\frac uv\Big)^4-2\Big(\frac uv\Big)^2+1\Big]=u^4-2u^2v^2+v^4
		\]
		and
		\[
		\psi_2^2(u,v)=4v^3\Big[\Big(\frac uv\Big)^3+\frac uv\Big]=4u^3+4uv^2.
		\]
	\end{example}
	Let $P$ be a non-torsion point in $E(\Q)$. There exist two integers $u$ and $v$ so that 
	\[
	x(P)=\frac uv \text{  with  } (u,v)=1 \text{  and  } v>0
	\]
	and then
	\begin{equation}\label{defxnp}
		x(nP)=\frac{\phi_n(x(P))}{\psi_n^2(x(P))}=\frac{v^{n^2}\phi_n(x(P))}{v^{n^2}\psi_n^2(x(P))}=\frac{\phi_n(u,v)}{v\psi_n^2(u,v)}.
	\end{equation} 
	For $n>0$, define \begin{equation}\label{definizgn}
		g_n(P)\coloneqq\gcd(\phi_n(u,v),v\psi_n^2(u,v))>0.
	\end{equation}
	Moreover, put
	\[
	g_0(P)=1.
	\]
	Observe that, if $\nu\in M_{\Q}^0$, then 
	\[
	\nu(g_n(P))=g_{n,\nu}(P)
	\]
	where $g_{n,\nu}(P)$ is defined in Definition \ref{defgnu}.
	
	Recall that $B_n(E,P)$, as defined in the introduction, represents the square root of the denominator of $x(nP)$ and that $\beta_n(E,P)=B_n(E,P)/B_1(E,P)$. Observe that $B_1^2(E,P)=v$. So, using (\ref{defxnp}),
	\begin{equation}\label{bngn}
		\beta_n^2(E,P)=\frac{B_n^2(E,P)}{B_1^2(E,P)}=\frac{v\psi_n^2(u,v)}{B_1^2(E,P)g_n(P)}=\frac{\psi_n^2(u,v)}{g_n(P)}.
	\end{equation}
	
	\begin{proposition}\label{gmn}
		Let $E$ be an elliptic curve defined by a Weierstrass equation with integer coefficients. Let $P\in E(\Q)$ be a non-torsion point.
		Let $m,n\in \N$. If $m$ is a multiple of $M(P)$, as defined in Definition \ref{MP}, then
		\[
		g_{n+m}(P)g_{\abs{n-m}}(P)=g_n^2(P)g_m^2(P).
		\]
	\end{proposition}
	\begin{proof}
		If $p$ divides $v$, then
		\[
		\phi_k(u,v)\equiv u^{k^2}\not\equiv 0\mod p
		\]
		since $\phi_k$ is monic from Lemma \ref{propdivpol2} and $(u,v)=1$. So, $\ord_p(g_k(P))=0$ for every $k\geq 1$. Suppose now that $p$ does not divide $v$. If $\nu$ is the place associated with $p$, then $\nu(x(P))\geq 0$ and
		\[
		\ord_p(g_k(P))=g_{k,\nu}(P).
		\]
		Therefore, we conclude using Proposition \ref{gcd}.
	\end{proof}
	\section{Proof of Theorem \ref{main}}\label{secmain}
	Recall that a sequence of integers is an EDSA if it is a sequence as in Definition \ref{defward} and it is an EDSB if it is a sequence as in Definition \ref{defeds}.
	
	Let $E$ be a rational elliptic curve defined by a Weierstrass equation with integer coefficients and let $P\in E(\Q)$. Let $x(P)=u/v$ with $u$ and $v$ coprime integers. Recall that $v$ is a square and let $v^{1/2}$ be the positive square root of $v$. Define
	\[
	h_n\coloneqq v^{\frac{n^2-1}{2}}\psi_n(x(P),y(P)),
	\]
	where $\psi_n$ is defined in Definition \ref{defdivpol}.
	This is a sequence of integers. As is shown in \cite[Exercise 3.7.g]{arithmetic}, the sequence of the $h_n$ is an EDSA. 
	
	Almost every EDSA is a sequence of the $h_n$ for some elliptic curve $E$ and a point $P$. Indeed, we have the following theorem, due to Ward.
	\begin{theorem}\cite[Theorem 12.1]{ward}
		Let $\{h_n\}_{n\in \N}$ be a non-singular EDSA with $h_2h_3\neq 0$. So, there exists an elliptic curve $E$ and a point $P\in E(\Q)$ such that, if we put $x(P)=u/v$ with $u$ and $v$ coprime integers and $v>0$, then
		\[
		h_n=v^{\frac{n^2-1}{2}}\psi_n(x(P),y(P)).
		\]
	\end{theorem}
	For a definition of non-singular EDSA, see \cite[Section 19]{ward}.
	
	Let $h_n$ be an EDSA as in the previous theorem and take $E$ and $P$ as in the theorem. Hence, for (\ref{bngn}),
	\[
	h_n^2=v^{n^2-1}\psi_n^2(x(P))=\frac{B_n^2(E,P)g_n(P)}{B_1^2(E,P)}=\beta_n^2(E,P)g_n(P).
	\]
	Observe that $h_n$ and $\beta_n(E,P)$ have the same sign and then
	\[
	h_n=\beta_n(E,P)\sqrt{g_n(P)}.
	\]
	So, given a non-singular EDSA, every term of the sequence is equal to the term of the EDSB $\{\beta_n(E,P)\}_{n\in \N}$, multiplied by $\sqrt{g_n(P)}$.
	
	Now, we want to show that every EDSB $\beta_n(E,P)$ is similar to an EDSA.
	Recall that $g_n(P)$ is defined in Equation (\ref{definizgn}) and that the sequence $\{\beta_n(E,P)\}_{n\in \N}$ is defined in Definition \ref{defeds}. Observe that $g_n(P)$ is a square thanks to (\ref{bngn}).
	
	\begin{lemma}\label{beta}
		Let $E$ be a rational elliptic curve defined by a Weierstrass equation with integer coefficients. Let $P\in E(\Q)$ with $x(P)=u/v$ for $u$ and $v$ two coprime integers and with $v>0$. Let $w=\sqrt{v}>0$ and define the sequence 
		\begin{equation}\label{defhn}
			h_n=w^{n^2-1}\psi_n(x(P),y(P)).
		\end{equation}
		Since $g_n(P)$ is a square, define $\sqrt{g_n(P)}$ as the positive square root of $g_n(P)$. Then,
		\begin{equation}\label{bhg}
			\beta_n(E,P)=\frac{h_n}{\sqrt{g_n(P)}}.
		\end{equation}
	\end{lemma} 
	\begin{proof}
		Observe that 
		\[
		h_n^2=v^{n^2-1}\psi_n^2(x(P))=\psi_n^2(u,v)\in \Z.
		\]
		Then, the sequence of the $h_n$ is a sequence of integers.
		Moreover, by definition, the sign of $\beta_n(E,P)$ agrees with the sign of $h_n$. Using (\ref{bngn}),
		\[
		\frac{h_n^2}{g_n(P)}=\frac{\psi_n^2(u,v)}{g_n(P)}=\Big(\frac{B_n(E,P)}{B_1(E,P)}\Big)^2=\beta_n^2(E,P).
		\]	
		Taking the square root, we conclude.
	\end{proof}
	In general, the sequence of the $\beta_n$ is not an EDSA, as we showed in Example \ref{ex2}. Recall that, as we said in Definition \ref{defward}, given an EDSA $\{h_n\}_{n\in \N}$, we have
	\[
	h_{m+n}h_{n-m}h_{r}^2=h_{n+r}h_{n-r}h_{m}^2-h_{m+r}h_{m-r}h_{n}^2
	\]
	for all $n\geq m\geq r$. We will show an EDSB satisfies a subset of these equations.
	
	Now, we are ready to prove Theorem \ref{main}. Recall that $M(P)$ is defined in Definition \ref{MP}. 
	\begin{proof}[Proof of Theorem \ref{main}]
		As is shown in \cite[Exercise 3.7.g]{arithmetic}, the sequence  
		\[
		\psi_n\coloneqq\psi_n(x(P),y(P))
		\]
		satisfies
		\[
		\psi_{n+m}\psi_{n-m}\psi_r^2=\psi_{n+r}\psi_{n-r}\psi_m^2-\psi_{m+r}\psi_{m-r}\psi_n^2 \quad \text{    for all    }n\geq m\geq r.
		\]
		Therefore, using the definition of $h_n$ in Lemma \ref{beta},
		\begin{align*}
			h_{n+m}h_{n-m}h_r^2=\nonumber&w^{(n+m)^2-1}\psi_{n+m}w^{(n-m)^2-1}\psi_{n-m}w^{2r^2-2}\psi_r^2\\=\nonumber&w^{2n^2+2m^2+2r^2-4}(\psi_{n+m}\psi_{n-m}\psi_{r}^2)\\=\nonumber&w^{2n^2+2m^2+2r^2-4}(\psi_{n+r}\psi_{n-r}\psi_m^2-\psi_{m+r}\psi_{m-r}\psi_n^2)\\=\nonumber&w^{(n+r)^2-1}\psi_{n+r}w^{(n-r)^2-1}\psi_{n-r}w^{2m^2-2}\psi_m^2\\\nonumber&-w^{(m+r)^2-1}\psi_{m+r}w^{(m-r)^2-1}\psi_{m-r}w^{2n^2-2}\psi_n^2\\=&h_{n+r}h_{n-r}h_m^2-h_{m+r}h_{m-r}h_n^2
		\end{align*}
		for all $n\geq m\geq r$. So, we have 
		\begin{equation}\label{hedsa}
			h_{n+m}h_{n-m}h_r^2=h_{n+r}h_{n-r}h_m^2-h_{m+r}h_{m-r}h_n^2
		\end{equation}
		and then, dividing both sides by $g_r(P)g_n(P)g_m(P)$ we obtain
		\[
		\frac{h_{n+m}h_{n-m}}{g_n(P)g_m(P)}\cdot\frac{h_r^2}{g_r(P)}=\frac{h_{n+r}h_{n-r}}{g_n(P)g_r(P)}\cdot\frac{h_m^2}{g_m(P)}-\frac{h_{m+r}h_{m-r}}{g_m(P)g_r(P)}\cdot\frac{h_n^2}{g_n(P)}.
		\]
		Define,
		\[
		L_{n,m}\coloneqq\frac{g_{m+n}(P)g_{n-m}(P)}{g_n^2(P)g_m^2(P)}.
		\]
		Using Lemma \ref{beta}, we substitute $h_n$ with $\beta_n\sqrt{g_n(P)}$ and we obtain
		\[
		\beta_{n+m}\beta_{n-m}\sqrt{L_{n,m}}\beta_r^2=\beta_{n+r}\beta_{n-r}\beta_m^2\sqrt{L_{n,r}}-\beta_{m+r}\beta_{m-r}\beta_n^2\sqrt{L_{m,r}}.
		\]
		Now, we conclude using Proposition \ref{gmn} since $L_{n,m}=L_{m,r}=L_{n,r}=1$.
	\end{proof}
	\begin{corollary}
		Let $E$ be a rational elliptic curve defined by a Weierstrass equation with integer coefficients. Let $P\in E(\Q)$ be a non-torsion point that is non-singular modulo every prime. Then the EDSB $\{\beta_n(E,P)\}_{n\in\N}$ is an EDSA. In particular, given a non-torsion point $P$, there exists a multiple $Q$ of $P$ such that $\{\beta_n(E,Q)\}_{n\in\N}$ is an EDSA.
	\end{corollary}
	\begin{proof}
		If $P$ is non-singular modulo every prime, then $M(P)=1$ and we are done using Theorem \ref{main}.
		Observe that the point $Q=M(P)P$ is non-singular modulo every prime. Hence, the sequence $\{\beta_n(E,Q)\}_{n\in\N}$ is an EDSA, using the first part of the corollary.
	\end{proof}
	\begin{remark}\label{rem}
		The previous corollary is equivalent to Theorem \ref{sh}. Anyway, the two proofs of this result are completely different. Our theorem is a generalization since it studies also the case when $M(P)\neq 1$.
	\end{remark}
	\begin{remark}
		In general, we cannot replace the constant $M(P)$ with a smaller positive integer in Theorem \ref{main}. We show an example of this fact. 
		
		Let $E$ be the elliptic curve defined by the equation $y^2=x^3+x+6$ and take $P=(-1,2)$. Consider the EDSB $\{\beta_n\}_{n\in \N}=\{\beta_n(E,P)\}_{n\in \N}$. The point $P$ is non-singular modulo every prime except modulo $2$. As we showed in Remark \ref{16}, $r(P,2)=3$ and then $M(P)=3$. Thanks to Theorem \ref{main}, we know that 
		\begin{equation}\label{eqrem}
			\beta_{n+m}\beta_{n-m}\beta_r^2=\beta_{n+r}\beta_{n-r}\beta_m^2-\beta_{m+r}\beta_{m-r}\beta_n^2
		\end{equation}
		if at least two of the indexes are multiples of $3$. In order to show that we cannot replace $M(P)$ with a smaller constant, we just need to show that the equation does not hold for $r=2$, $m=4$, and $n=6$. Using the definition, we compute that $\beta_2=1$, $\beta_4=-3$, $\beta_6=8$, $\beta_8=-93$, and $\beta_{10}=463$. Hence, Equation (\ref{eqrem}) does not holds for these values and then we cannot replace $M(P)$ with $1$ or $2$ in Theorem \ref{main}.
	\end{remark}
	Now, we briefly deal with the problem when $P$ is a torsion point.
	\begin{corollary}\label{remtors}
		Let $E$ be a rational elliptic curve defined by a Weierstrass equation with integer coefficients and let $P\in E(\Q)$ be a torsion point. Assume that $M(P)=1$. Consider the EDSB $\{\beta_n\}_{n\in \N}=\{\beta_n(E,P)\}_{n\in \N}$. For all $n\geq m\geq r$,
		\[
		\beta_{n+m}\beta_{n-m}\beta_r^2=\beta_{n+r}\beta_{n-r}\beta_m^2-\beta_{m+r}\beta_{m-r}\beta_n^2.
		\]
	\end{corollary}
	\begin{proof}
		Using Proposition \ref{ayad}, we have that $g_n(P)=1$ for every $n\in \N$. Hence, using (\ref{bhg}),
		\[
		\beta_n=h_n.
		\]
		As we proved in (\ref{hedsa}), we have \[
		h_{n+m}h_{n-m}h_r^2=h_{n+r}h_{n-r}h_m^2-h_{m+r}h_{m-r}h_n^2.
		\]
		for all $n\geq m\geq r$. Hence, for all $n\geq m\geq r$,
		\[
		\beta_{n+m}\beta_{n-m}\beta_r^2=\beta_{n+r}\beta_{n-r}\beta_m^2-\beta_{m+r}\beta_{m-r}\beta_n^2.
		\]
	\end{proof}
	If one wants to deal with the problem when $P$ is a torsion point and $M(P)\neq 1$, then it is necessary to obtain an analogue of Theorem \ref{ch} in the case when $P$ is a torsion point.
	
	Now, we briefly deal with the case when $E$ is not defined by a Weierstrass equation with integer coefficients.
	Let $E$ be a rational elliptic curve and let $P$ be a non-torsion point on $E(\Q)$. Define $\beta_n=\beta_n(E,P)$ as in Definition \ref{defeds}, i.e.
	\[
	\beta_n=\sign(\psi_n(x(P),y(P)))\cdot \frac{B_n(E,P)}{B_1(E,P)}.
	\] 
	With $B_n(E,P)$ we denote the positive square root of the denominator of $x(nP)$. As we said before, in general, the sequence $\beta_n(E,P)$ is not a sequence of integers. For example, if $E$ is defined by the equation $y^2=x^3+7^{-4}x+7^{-6}$ and $P=(0,7^{-3})$, then $\beta_4(E,P)=-36\sqrt{7}$. Anyway, we can find an analogue of Theorem \ref{main} even in the case when $E$ is not defined by a Weierstrass equation with integer coefficients. 
	\begin{proposition}\label{nonint}
		Let $E$ be a rational elliptic curve and let $P$ be a non-torsion point on $E(\Q)$. Define $\beta_n=\beta_n(E,P)$ as before. There exists a constant $M(P)$ such that, if $n\geq m\geq  r$ are three positive integers such that two of them are multiples of $M(P)$, then
		\[
		\beta_{n+m}\beta_{n-m}\beta_r^2=\beta_{n+r}\beta_{n-r}\beta_m^2-\beta_{m+r}\beta_{m-r}\beta_n^2.
		\]
	\end{proposition}
	\begin{proof}
		It is easy to show that there exists an elliptic curve $E'$, defined by a Weierstrass equation with integer coefficients, and an isomorphism $\varphi:E'\to E$ in the form $\varphi(x',y')=(x'/u^2,y'/u^3)$ for $u\in \Z_>0$. Let $P'=\varphi^{-1}(P)$. So,
		\[
		\frac{A_n(E,P)}{B_n^2(E,P)}=x(nP)=\frac{x(nP')}{u^2}=\frac{A_n(E',P')}{u^2B_n^2(E',P')}
		\]
		and then
		\[
		B_n(E,P)=\frac{uB_n(E',P')}{\sqrt{\gcd(A_n(E',P'),u^2)}}.
		\]
		Hence, using the definition of $\beta_n$, we have
		\begin{equation}\label{ck}
			\beta_n(E,P)=\beta_n(E',P')\frac{B_n(E,P)}{B_n(E',P')}\frac{B_1(E',P')}{B_1(E,P)}=\beta_n(E',P')\frac{u}{\sqrt{\gcd(A_n(E',P'),u^2)}}\frac{B_1(E',P')}{B_1(E,P)}.
		\end{equation}
		Put $\beta_n=\beta_n(E,P)$ and $\beta_n'=\beta_n(E',P')$.
		Since $E'$ is defined by a Weierstrass equation with integer coefficients, there is a constant $M(P')$ such that if $n\geq m\geq  r$ are three positive integers such that two of them are multiples of $M(P')$, then
		\[
		\beta_{n+m}'\beta_{n-m}'\beta_r'^2=\beta_{n+r}'\beta_{n-r}'\beta_m'^2-\beta_{m+r}'\beta_{m-r}'\beta_n'^2.
		\]
		This follows from Theorem \ref{main}.
		
		Let $n_u$ be the smallest positive index such that $B_{n_u}(E',P')\equiv 0\mod{u}$. Let $k$ and $j$ be two positive integers. Observe that \[\gcd\Big(A_k(E',P'),u^2\Big)=\gcd\Big(A_j(E',P'),u^2\Big)\] if $k\equiv j\mod{n_u}$ or $k\equiv -j\mod{n_u}$. This follows easily form the fact that $x(kP')=x(-kP')$ and from the explicit formula to compute the abscissa of the sum of two points.
		
		Let $c_k$ be such that $\beta_k=c_k\beta_k'$, that can be computed using (\ref{ck}). For the previous observation we have $c_k=c_j$ if $k\equiv j\mod{n_u}$ or $k\equiv -j\mod{n_u}$. Put \[M(P)=\lcm(M(P'),n_u).\]Let $n\geq m\geq r$ and assume that $n$ and $m$ are both multiples of $M(P)$. Then, by definition, $M(P)$ is a multiple of $M(P')$ and it is a multiple of $n_u$. Therefore, $c_{n_u}=c_n=c_m=c_{m\pm n}$ and $c_r=c_{m\pm r}=c_{n\pm r}$. Moreover since $M(P)$ is a multiple of $M(P')$, we have
		\[
		\beta_{n+m}'\beta_{n-m}'\beta_r'^2=\beta_{n+r}'\beta_{n-r}'\beta_m'^2-\beta_{m+r}'\beta_{m-r}'\beta_n'^2.
		\]
		Hence,
		\begin{align*}
			\beta_{n+m}\beta_{n-m}\beta_r^2&=c_{n+m}c_{n-m}c_{r}^2	\beta_{n+m}'\beta_{n-m}'\beta_r'^2\\&=c_{n_u}^2c_{r}^2\beta_{n+r}'\beta_{n-r}'\beta_m'^2-c_{n_u}^2c_{r}^2\beta_{m+r}'\beta_{m-r}'\beta_n'^2\\&=c_{m}^2c_{n+r}c_{n-r}\beta_{n+r}'\beta_{n-r}'\beta_m'^2-c_{n}^2c_{m+r}c_{m-r}\beta_{m+r}'\beta_{m-r}'\beta_n'^2\\&=\beta_{n+r}\beta_{n-r}\beta_m^2-\beta_{m+r}\beta_{m-r}\beta_n^2.
		\end{align*}
		The cases when $n$ and $r$ are both multiples of $M(P)$ or when $r$ and $m$ are both multiples of $M(P)$ are identical. So, we have
		\[
		\beta_{n+m}\beta_{n-m}\beta_r^2=\beta_{n+r}\beta_{n-r}\beta_m^2-\beta_{m+r}\beta_{m-r}\beta_n^2
		\]
		and we are done.
	\end{proof}

	\normalsize
	\baselineskip=17pt
	\bibliographystyle{siam}
	\bibliography{biblio}

\begin{thebibliography}{1}

\bibitem{Ayad}
{\sc M.~Ayad}, {\em Points {$S$}-entiers des courbes elliptiques}, Manuscripta
  Math., 76 (1992), pp.~305--324.

\bibitem{chhadiv}
{\sc J.~Cheon and S.~Hahn}, {\em Explicit valuations of division polynomials of
  an elliptic curve}, Manuscripta Math., 97 (1998), pp.~319--328.

\bibitem{naskrkecki2022common}
{\sc B.~Naskr{\k{e}}cki and M.~Verzobio}, {\em Common valuations of division
  polynomials}, Proceedings of the Royal Society of Edinburgh Section A:
  Mathematics,  (2022), pp.~1--15.

\bibitem{shipsey}
{\sc R.~Shipsey}, {\em Elliptic Divisibility Sequences}, PhD thesis, Goldsmiths
  College, London University, 2000.

\bibitem{arithmetic}
{\sc J.~H. Silverman}, {\em The arithmetic of elliptic curves}, vol.~106 of
  Graduate Texts in Mathematics, Springer, Dordrecht, second~ed., 2009.

\bibitem{yv21}
{\sc P.~Voutier and M.~Yabuta}, {\em The greatest common valuation of $\phi_n$
  and $\psi_n^2$ at points on elliptic curves}, Journal of Number Theory, 229
  (2021), pp.~16--38.

\bibitem{ward}
{\sc M.~Ward}, {\em Memoir on elliptic divisibility sequences}, Amer. J. Math.,
  70 (1948), pp.~31--74.

\end{thebibliography}
	IST Austria, Am Campus 1, 3400 Klosterneuburg, Austria\\
	\textit{E-mail address}: matteo.verzobio@gmail.com
\end{document}